\documentclass[12pt,a4paper]{article}
\usepackage{amssymb,comment,amsmath,amsthm}
\usepackage{graphicx}
\usepackage{color}
\usepackage{tikz}
\usetikzlibrary{positioning}
\usetikzlibrary{decorations.pathreplacing}
\usepackage{titling,titlesec}
\usepackage{url}
\usepackage{lipsum}
\titleformat{\subsection}[runin]
            {\normalfont\bfseries}{}{0em}{}[.]

\newfont{\blb}{msbm10 scaled\magstep1}
\newfont{\comp}{cmr12 scaled\magstep1}
\newfont{\compb}{cmr10 scaled\magstep2}
\newfont{\sbb}{cmssbx10 scaled\magstep3}
\newfont{\sbbb}{cmssbx10 scaled\magstep5}
\newfont{\sbs}{cmssbx10 scaled\magstep1}
\newtheorem{theorem}{Theorem}
\newtheorem{lemma}{Lemma}
\newtheorem{claim}[subsection]{Claim}

\newtheorem{corollary}{Corollary}
\newtheorem{definition}{Definition}

\def\cL{\mathcal{L}}

\title{Combining the theorems of Tur\'an  and de Bruijn-Erd\H os}

\author{
Sayok Chakravarty\thanks{Department of Mathematics, Statistics and Computer Science, University of Illinois, Chicago, IL 60607. Email: schakr31@uic.edu. Research partially supported by NSF Award DMS-2153576.  }\and
Dhruv Mubayi\thanks{Department of Mathematics, Statistics and Computer Science, University of Illinois, Chicago, IL 60607. Email: mubayi@uic.edu. Research partially supported by NSF Awards
DMS-1952767 and DMS-2153576 and a Simons Fellowship.}
}

\begin{document}

\maketitle

\begin{abstract}
     Fix an integer $s \ge 2$. Let $\mathcal{P}$ be a set of $n$ points and let $\mathcal{L}$ be a set of lines in a linear space  such that no line in $\mathcal{L}$ contains more than $(n-1)/(s-1)$  points of $\mathcal{P}$. Suppose that for every $s$-set $S$ in $\mathcal{P}$,  there is a pair of points in   $S$ that lies in a line from $\mathcal{L}$. We prove that  $|\cL| \ge (n-1)/(s-1)+s-1$ for $n$ large, and this is sharp when $n-1$ is a multiple of $s-1$. This generalizes the de Bruijn-Erd\H os theorem which is the case $s=2$. Our result is proved in the more general setting of linear hypergraphs.
\end{abstract}

\section{Introduction}
A finite linear space over a set $X$ is a family $\mathcal{L}$ of its subsets, called lines,
such that every line contains at least two points, and any two points are on exactly one line. A fundamental theorem proved by  de Bruijn and Erd\H os~\cite{MR28289} states that if $\mathcal{L}$ is a finite linear space over a set $X$ with $X \notin \mathcal{L}$, then $|\mathcal{L}| \ge |X|$ and equality holds if and only if $\mathcal{L}$ is either a near pencil or a projective plane. This is often viewed as a statement in incidence geometry, in which case it states that the number of lines determined by $n$ points in a projective plane is at least $n$. The result also has the following graph theoretic formulation: the minimum number of proper complete subgraphs of the complete graph $K_n$ that are needed to partition its edge set is $n$ (see~\cite{alon2012bruijn} for an extension of this formulation to hypergraphs). Various other extensions have been studied. For instance,~\cite{chiniforooshan2011bruijn} considered the problem of determining the minimum number of lines determined by $n$ points in general metric spaces  and~\cite{dolevzal2019bruijn} defined a notion of de Bruijn-Erd\H os sets in measure spaces and bounded the Hausdorff dimension and Hausdorff measure of such sets. The de Bruijn-Erd\H os theorem is also a basic result in extremal set theory and design theory that has many extensions and generalizations. The most  notable of these are due to Fisher~\cite{Fisher}, Bose~\cite{bose}, and Ray-Chaudhuri-Wilson~\cite{RW}.  

Here we consider another natural generalization of the  de Bruijn-Erd\H os theorem.  We relax the condition that every pair of points lies in a line as follows. An $s$-set is a set of size $s$.

\begin{definition}
A collection of subsets (lines) $\mathcal{L}$ of a set $X$ is an $s$-cover  if every two lines in 
$\cL$ have at most one point in common and for every $s$-set $S \subset X$, some pair of points from $S$ lies in a line in $\cL$. 
\end{definition}

 Note that when $s=2$, this is the definition of a linear space (excluding trivial requirements). We can view this definition through the lens of graph theory as follows. 
 Consider the graph $G=(X,E)$ where $E$ is the set of pairs not contained in any line in $\mathcal{L}$. Then $G$ is $K_s$-free when $\mathcal{L}$ is an $s$-cover. Hence, the $s$-cover condition can be thought of as a Tur\'an-type property. 
 
  As $s$ becomes larger, the requirement for a family to be an $s$-cover becomes weaker, and hence the number of lines needed for an $s$-cover decreases. So a natural question is to ask for the size of a smallest $s$-cover.  In order to make this problem nontrivial, we need to impose an upper bound on the size of subsets in $\cL$. For example, if we allow sets of size $|X|$, then just one set suffices to cover every pair. Morever, if  sets in $\cL$ are allowed to be of size greater than $(n-1)/(s-1)$, then we can take a collection of $s-1$ pairwise disjoint sets that cover all the points. This is an $s$-cover, as any $s$ points will contain two points in one of the sets and will be covered. As $n \rightarrow\infty $ this is has constant size.
  Hence the natural condition to obtain a nontrivial result as $n \rightarrow\infty $ is that all sets in $\cL$ have size at most $t=(n-1)/(s-1)$.  
  
  Under this condition, a straightforward construction reminiscient of the construction for Tur\'an's graph theorem is the following. Assume that $t=(n-1)/(s-1)$ is an integer. The underlying set is a $t \times (s-1)$ grid with an additional new vertex $z$, and the line set comprises all  columns as well as  all rows where we append $z$ to each row (see Figure \ref{picture}). Formally, $X= ([t] \times [s-1]) \cup \{z\}$, and 
  $$\mathcal L = \{c_i: 1\le i \le s-1\} \cup \{r_j \cup \{z\}: 1\le j \le t\},$$ where the $i$th column is $c_i:=[t] \times \{i\}$ and the $j$th row is $r_j:=\{j\} \times [s-1]$.  
   This yields an $s$-cover with $|\mathcal L| = s-1 + t$. 
   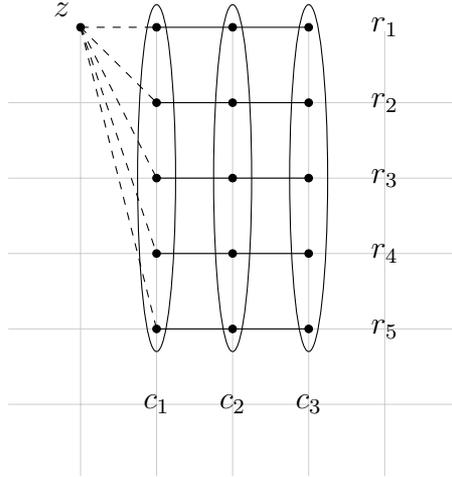
\begin{figure} \label{picture}
   \begin{center}
      \begin{tikzpicture}

        \draw[step=1,help lines,black!20] (-0.95,-0.95) grid (4.95,4.95);
        
        \foreach \Point/\PointLabel in {(0,5)/z, (1,1)/, (1,2)/, (1,3)/, (1,4)/, (1,5)/, (2,1)/, (2,2)/, (2,3)/,(2,4)/, (2,5)/, (3,1)/, (3,2)/, (3,3)/, (3,4)/, (3,5)/}
        \draw[fill=black] \Point circle (0.05) node[above left] 
        {$\PointLabel$};

        \draw (1,1)--(3,1) node at (4,5) {$r_1$};
        \draw (1,2)--(3,2) node at (4,4) {$r_2$};
        \draw (1,3)--(3,3) node at (4,3) {$r_3$};
        \draw (1,4)--(3,4) node at (4,2) {$r_4$};
        \draw (1,5)--(3,5) node at (4,1) {$r_5$};

        \draw[dashed](0,5)--(1,1) node[above]{};
        \draw[dashed] (0,5)--(1,2) node[above]{};
        \draw[dashed] (0,5)--(1,3) node[above]{};
        \draw[dashed] (0,5)--(1,4) node[above]{};
        \draw[dashed] (0,5)--(1,5) node[above]{};

        \draw (1,3) ellipse (0.25cm and 2.3cm) node at (1,0)  {$c_1$};
        \draw (2,3) ellipse (0.25cm and 2.3cm) node at (2,0)  {$c_2$};
        \draw (3,3) ellipse (0.25cm and 2.3cm) node at (3,0)  {$c_3$};
    \end{tikzpicture}  
    \end{center}
    \caption{The construction of $\cL$ when $t=5$ and $s=4$.}
   \end{figure}

 In this paper, we show that the above construction is tight. Our main result is the following theorem.
 
\begin{theorem} \label{mainthm}
Fix $s \ge 2$. Let $\mathcal{L}$ be an $s$-cover over a set of size $n$. Suppose that each set in $\cL$ has size at most $(n-1)/(s-1)$. Then $|\cL| \ge (n-1)/(s-1)+s-1$ for $n$ large and this is tight if  $(s-1) \mid (n-1)$. If $(s-1) \nmid (n-1)$, then the bound $(n-1)/(s-1)+s-1$ is tight asymptotically as $n \rightarrow \infty$. 
\end{theorem}

 The general framework of Theorem~\ref{mainthm} specializes to give appealing geometric statements as given in the abstract or the more special form below.

\begin{corollary}
 Fix an integer $s \ge 2$. Let $\mathcal{P}$ be a set of $n$ points and let $\mathcal{L}$ be a set of $m$ lines in the plane  such that no line in $\mathcal{L}$ contains more than $(n-1)/(s-1)$  points of $\mathcal{P}$. Suppose that for every $s$-set $S$ of points from $\mathcal{P}$,  there is a  pair of points in $S$ lies in some line from $\mathcal{L}$. Then $m \ge (n-1)/(s-1) + s-1$ for $n$ large and if $n-1$ is a multiple of $s-1$,  this is tight. 
\end{corollary}
We remark that equality can hold above as some of the hypergraphs we construct to prove tightness for Theorem \ref{mainthm} can be realized as lines in the plane.

Our proof requires $n$ to be large in terms of $s$ and it remains an open problem to prove the result for small $n$. 

\section{Proof of Theorem~\ref{mainthm}}
We will prove Theorem~\ref{mainthm} by induction on $s$. However, in order to facilitate the induction argument, we need to prove a slightly stronger statement for $s \ge 3$ as shown below.

\begin{theorem} \label{mainthmstrong}
The statement of Theorem~\ref{mainthm} holds with the following strengthening. If $s \ge 3$ and each set in $\cL$ has size at most $(n-1)/(s-1)-1$, then $|\cL|>(n-1)/(s-1)+s-1$.
\end{theorem}

\subsection{Notation and Lemmas}
Let $\mathcal{L}=\left\{A_1,A_2,\ldots, A_m \right\}$ be an $s$-cover on $X:=[n]$. Assume $|A_i|=a_i$ and $(n-1)/(s-1)\ge |A_1| \ge |A_2| \ge \cdots \ge |A_m|$. For $x \in X$, the degree of $x$, written $d(x)$, is the number of $A_i$ that contain $x$ and the neighborhood of $x$, written $N(x)$, is the collection of $A_i$'s that contain $x$. Let $d = \min_{w \in A_1} d(w)$ and suppose that $d(v)=d$ where $v \in A_1$. Let $\left\{A_{i_1},\ldots, A_{i_d} \right\}$ be the neighborhood of $v$ where $i_1=1$ and let $Q:=[n] \setminus \bigcup_{j=1}^d A_{i_j}$ and $P:=\bigcup_{j=1}^d A_{i_j}$. Set $p:=|P|$ (see Figure \ref{P and Q}). Throughout the proof, we say a subset $J \subset [n]$ is covered if there exists some $A_i$ containing $J$.

We first prove a lemma giving various bounds on $m$ depending on $d$, the $a_i$'s, and $|Q|$.  We will assume below that Theorem~\ref{mainthmstrong} holds for all $s^{\prime} \le s-1$ by induction on $s$ and that $n$ is sufficiently large in terms of $s$ to apply induction and any further inequalities that require this. More explicitly, we will show that if  $n \ge n_0(s)$ where $n_0(s)$ is large enough for the inequalities we use in the proof to hold,  then $|\mathcal{L}| \ge (n-1)/(s-1)+s-1$.

The base case $s=2$ of the induction follows from the de Bruijn-Erd\H os theorem, so we assume $s>2$. We note that the stronger statement of Theorem~\ref{mainthmstrong} holds only for $s \ge 3$, and we will take care of the specific case $s=3$ during the proof.

Let $\delta,\delta_1, \delta_2, \delta_3,\delta_4 >0$ be constants that follow the hierarchy
$$\frac{1}{C_1^{1/10}} \ll \delta_2 \ll \delta_1 \ll \delta \ll \delta_4 \ll \delta_3  \ll \frac{1}{s^2}$$
 where $\xi \ll \eta $ simply means that $\xi$ is a sufficiently small function of $\eta$ that is needed to satisfy some inequality in the proof. In particular,  set
 $$\delta_3 =\delta^{1/4} \qquad \hbox{ and } \qquad \delta_4 = \delta^{1/2}.$$
 We will repeatedly use the fact that if $\beta>1/10$ and $0<\alpha<8\beta$, 
\begin{equation} \label{s/n bound}
    \frac{s^{\alpha}}{n^{\beta}} \le \frac{s^{\alpha}}{(C_1 s^8)^{\beta}} \le  \frac{s^{\alpha-8\beta}}{C_1^\beta} \le \frac{1}{C_1^{\beta}} \ll  \delta_2.
\end{equation}
   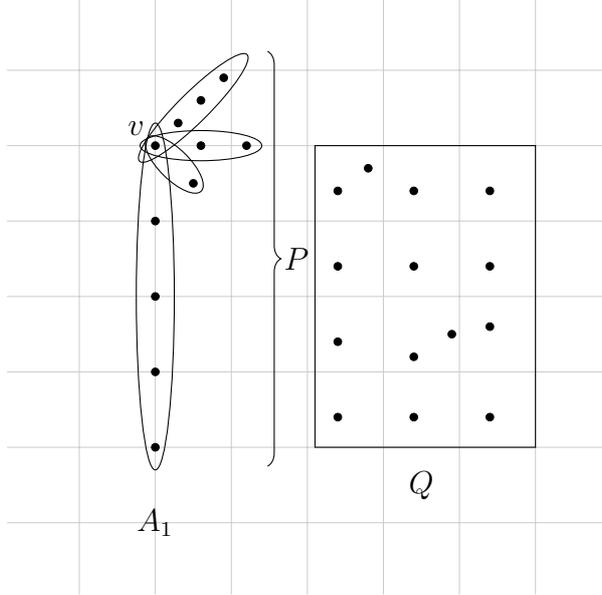
\begin{figure}
   \begin{center}
      \begin{tikzpicture}

        \draw[step=1,help lines,black!20] (-0.95,-0.95) grid (6.95,6.95);
        
        \foreach \Point/\PointLabel in {((1,1)/, (1,2)/, (1,3)/, (1,4)/, (1,5)/v,(1.3,5.3)/,(1.6,5.6)/,(1.9,5.9)/,(1.6,5)/,(2.2,5)/,(1.5,4.5)/,(3.4,1.4)/,(4.4,1.4)/,(5.4,1.4)/,(3.4,2.4)/,(4.4,2.2)/,(5.4,2.6)/,(3.4,3.4)/,(4.4,3.4)/,(5.4,3.4)/,(3.4,4.4)/,(4.4,4.4)/,(5.4,4.4)/,(3.8,4.7)/,(4.9,2.5)/}
        \draw[fill=black] \Point circle (0.05) node[above left]
        {$\PointLabel$};
        \draw (1,3) ellipse (0.25cm and 2.3cm) node at (1,0)  {$A_1$};
        \draw[rotate around={-45:(1.5,5.5)}] (1.5,5.5) ellipse (0.2cm and 1.0cm);
        \draw[rotate around={90:(1.6,5)}] (1.6,5) ellipse (0.2cm and 0.8cm);
        \draw[rotate around={45:(1.25,4.75)}] (1.25,4.75) ellipse (0.2cm and 0.5cm);
        \draw [decorate,decoration={brace,amplitude=5pt,mirror,raise=4ex}]
  (1.75,0.75) -- (1.75,6.25) node[midway,right=0.8]{$P$};
        \draw [draw=black] (6,5) rectangle (3.1,1) node at (4.5,0.5)  {$Q$};
    \end{tikzpicture}  
    \end{center}
    \caption{Setup for Lemma \ref{bounds}.} \label{P and Q}
   \end{figure}
   
\begin{lemma} \label{bounds}
The following bounds hold for $n \ge n_0(s)$.
\begin{enumerate}
\item $m \ge a_1 (d-1)+1$
 \item If $|Q|> n(s-2)/(s-1)$, then $$m \ge \frac{|Q|-1}{s-2}+s-2+d.$$  
\item  For  distinct $i_1,i_2,...,i_s$, let $A_{i_k}^{\prime} \subset A_{i_k}$ be pairwise disjoint subsets, $a_{i_k}^{\prime}=|A_{i_k}^{\prime}|$.  Then, 
$$m \ge \frac{a_{i_1}^{\prime}a_{i_2}^{\prime}a_{i_3}^{\prime}...a_{i_s}^{\prime}}{e_{s-2}(a_{i_1}^{\prime},a_{i_2}^{\prime},a_{i_3}^{\prime},\ldots ,a_{i_s}^{\prime})}$$ 
where $e_k(x_1,\ldots,x_n)=\sum_{1 \le i_1 < \cdots  <i_k \le n} x_{i_1}\cdots x_{i_k}$ is the $k$-th elementary symmetric polynomial. 

\item $\sum_{i=1}^m \binom{a_i}{2} \ge \frac{n^2}{2(s-1)}-\frac{n}{2}$. 
 \item Suppose $a_1 \ge (1-\delta) \cdot \sqrt{n}$. Write $n-1+(s-1)(s-2)=(s-1)a_1 q + r$ for integers $q,r$ with $0 \le r < (s-1)a_1$.
 If $d \notin \left\{q,q+1 \right\} $, then $m > \frac{n-1}{s-1}+s-1$.

\end{enumerate}
\end{lemma}

\begin{proof} We prove each statement of the lemma.

1.  The number of sets containing a vertex in $A_1$ is at least $a_1 (d-1)+1$ as $d$ is minimum degree of vertices in $A_1$.

 2.  Observe that  
 $$\frac{|Q|-1}{s-2}>\frac{\frac{n(s-2)}{s-1}-1}{s-2}=\frac{n}{s-1}-\frac{1}{s-2}$$
This implies $(|Q|-1)(s-1) \ge n(s-2)-(s-1)+1=(n-1)(s-2)$. Rearranging, we get $(n-1)/(s-1) \le (|Q|-1)/(s-2)$. This shows the size of all sets not in $N(v)$ is less than $(|Q|-1)/(s-2)$. Also,
 $$|Q| > \frac{s-2}{s-1}n \ge \frac{s-2}{s-1}n_0(s)
\ge n_0(s-1).$$
For any  $s-1$ distinct vertices  $x_1, \ldots, x_{s-1} \in Q$, there must be a set $A_t$ containing a pair from $x_1,\ldots, x_{s-1}$ as the $s$-set $\{x_1,\ldots, x_{s-1}, v\}$ must be covered and $x_1,\ldots,x_{s-1} \notin P$. Therefore the collection of sets $A_i \setminus P$ that have at least one point in $Q$ is an $(s-1)$-cover of $Q$.  By the induction hypothesis,  the number of these sets is at least $(|Q|-1)/(s-2)+s-2$. Hence $m \ge (|Q|-1)/(s-2)+s-2+d$ as there are $d$ sets containing $v$ in addition to these $A_i$.

3.  Consider the collection of $s$-sets  $B=\{\{x_1, \ldots, x_s\}: x_j \in A_{i_j}', j=1,\ldots, s\}$. A particular $A_t$ covers at most $e_{s-2}(a_{i_1}^{\prime},\ldots, a_{i_s}^{\prime})$ such $s$-sets as it has at most one point in each of $A_{i_1}$,\ldots, $A_{i_s}$. The number of $s$-sets in $B$ is $a'_{i_1}\cdots a_{i_s}'$. It follows that 
$$m \ge \frac{a_{i_1}' \cdots a_{i_s}'}{e_{s-2}(a_{i_1}^{\prime},\ldots ,a_{i_s}^{\prime})}.$$ 

4. Let $G=([n],E)$ be the graph where $E$ is the set of pairs not contained in any $A_i$. Then $G$ is $K_s$-free. Since every pair in $[n]$ is either in some $A_i$ or in $E$, we have 
$$\binom{n}{2}=\sum_{i=1}^m \binom{a_i}{2}+|E|$$
Since $|E| \le n^2/2 \cdot (1-1/(s-1))$ by Tur\'an's theorem, 
$$\sum_{i=1}^m \binom{a_i}{2}  \ge \binom{n}{2}-\frac{n^2}{2}\left(1 - \frac{1}{s-1}\right)=\frac{n^2}{2(s-1)}-\frac{n}{2}.$$

5. We have $n-1+(s-1)(s-2)=(s-1)a_1q+r$
for integers $q,r$ with $0 \le r< (s-1)a_1$. If $d \ge q+2$, then part 1 implies 
\begin{align*}
 m
 & \ge a_1(q+1)+1 \\
 & = \frac{n-1+(s-1)(s-2)-r}{s-1}+a_1+1 \\
 & = \frac{n-1}{s-1}+s-1+a_1-\frac{r}{s-1} \\
 & > \frac{n-1}{s-1}+s-1.
\end{align*}
If $d \le q-1$, then
\begin{align*}
|Q|
& \ge n-d(a_1-1)-1 \\
& \ge n-(q-1)(a_1-1)-1 \\
& = n-qa_1+q+a_1-1-1 \\
& = n-\frac{n-1+(s-1)(s-2)-r}{s-1}+q+a_1-2 \\
& > \frac{s-2}{s-1}n-(s-2)+q+a_1-2  \\
& = \frac{s-2}{s-1}n+q+a_1-s.
\end{align*}
Note that
$$a_1 - s \ge (1-\delta)\sqrt{n_0(s)}-s > s.$$

It follows that $q+a_1-s\ge s$, so 
$|Q|>n(s-2)/(s-1) +s.$
By part 2 $$m 
\ge \frac{|Q|-1}{s-2}+s-2 
\ge \frac{\frac{s-2}{s-1}n+ s}{s-2}+s-2  > \frac{n-1}{s-1}+s-1.$$
\end{proof}

{\bf Proof of Lower Bound for Theorem~\ref{mainthm}.} 

We prove Theorem 1 by considering different ranges for $a_1$.  Let $n \ge n_0(s)$ and set $\varepsilon=1/10s^2$.

{\bf Case 1:} $a_1 < (1-\delta)\sqrt n$.

 From Lemma~\ref{bounds}.4, we have $\sum_{i=1}^m \binom{a_i}{2} \ge n^2/2(s-1)-n/2$. Since $\sum_{i=1}^m \binom{a_i}{2} \le m a_1^2/2< m (1-\delta)^2n/2$, we have 
 $$\frac{(1-\delta)^2}{2}nm \ge \frac{n^2}{2(s-1)}-\frac{n}{2}.$$
This implies 
$$m \ge \frac{\frac{n^2}{2(s-1)}-\frac{n}{2}}{\frac{(1-\delta)^2}{2}n}=\frac{1}{(1-\delta)^2}\left( \frac{n}{s-1}-1\right).$$
For $n\ge n_0(s)$, this is greater than $(n-1)/(s-1)+s-1$.

{\bf Case 2:} $(1-\delta)\sqrt{n} \le a_1 \le 10 \sqrt{sn}$.

By Lemma~\ref{bounds}.5, we can assume $d$ is either $q$ or $q+1$. Suppose $|Q|>n\left(1-1/(s-1)\right)$. Then by Lemma~\ref{bounds}.2
 $$m 
\ge \frac{|Q|-1}{s-2}+s-2+d  > \frac{n}{s-1}-\frac{1}{s-2}+s-2+d.$$

Since 
$$\frac{n-1}{(s-1)a_1} < \frac{n-1}{(s-1)a_1}+\frac{s-2}{a_1}=q+\frac{r}{(s-1)a_1}<q+1,$$
we have
\begin{equation} \label{dlowerbound}
 d \ge q \ge \frac{n-1}{a_1(s-1)} -1\ge \frac{n-1}{10\sqrt{s}(s-1)\sqrt{n} } -1.
\end{equation}
Hence, for $n\ge n_0(s)$ 
$$m>\frac{n}{s-1}-\frac{1}{s-2}+s-2+d  > \frac{n-1}{s-1}+s-1.$$
Therefore $|Q|\le n (1-1/(s-1))$ and $p:=|P|=n-|Q| \ge n/(s-1)$. By Turan's theorem, at least $p^2/2(s-1)-p/2$ of the pairs in $P$ must be covered. The number of pairs covered by  $A_{i_1},\ldots,A_{i_d}$ is
$$ \sum_{j=1}^d \binom{a_{i_j}}{2} \le d \binom{a_1}{2} \le \frac{a_1^2 d}{2}.$$
Since a set $A_i$ in $\mathcal{L} \setminus N(v)$ has at most $d$ points from $P$, it covers at most $\binom{d}{2}$ pairs. So
    \begin{equation} \label{eqnmbound}m \ge \frac{\frac{p^2}{2(s-1)}-\frac{p}{2}-\sum_{j=1}^d \binom{a_{i_j}}{2}}{\frac{d^2}{2}}\ge\frac{p^2}{(s-1)d^2} \left(1 -\frac{s-1}{p} -\frac{a_1^2 d (s-1)}{p^2}\right).\end{equation}

Note that
$$\frac{a_1^2 d (s-1)}{p^2} \le \frac{100s n d(s-1)}{n^2/(s-1)^2} = \frac{100ds(s-1)^3}{n}.$$
Observe that 
\begin{equation}\label{eqndbound}d \le q+1 \le \frac{n-1}{a_1(s-1)}+\frac{s-2}{a_1}+1\end{equation}
and 
\begin{equation} \label{eqndboundsimplified}
     \begin{split}
       \frac{n-1}{a_1(s-1)}+\frac{s-2}{a_1}+1
       & \le \frac{n-1}{(1-\delta)\sqrt{n}(s-1)}+\frac{s-2}{(1-\delta)\sqrt{n}}+1 \\
       & \le \frac{\sqrt{n}}{(1-\delta)(s-1)}+\delta_2+1.
     \end{split}
\end{equation}
Consequently, by (\ref{eqndbound}),
$$d(s-1) \le \frac{\sqrt{n}}{1-\delta}+(s-1)(1+\delta_2).$$
This and (\ref{s/n bound}) imply
\begin{align*}
\frac{100ds(s-1)^3}{n}
& \le \frac{100s^3}{n} d(s-1) \\
& \le \frac{100s^3}{n}\left(\frac{\sqrt{n}}{1-\delta}+(s-1)(1+\delta_2) \right) \\
& =\frac{100s^3}{\sqrt{n}(1-\delta)}+\frac{100s^3(s-1)(1+\delta_2)}{n} \\
& \le \frac{\delta_2}{2}.
\end{align*}
It follows that $a_1^2 d (s-1)/p^2 \le \delta_2/2$. As $p\ge n/(s-1)$, we have $(s-1)/p \le (s-1)^2/n \le s^2/n\le \delta_2/2 $ by (\ref{s/n bound}) and hence by (\ref{eqnmbound})
\begin{equation} \label{mlowerbound}
m \ge (1-\delta_2)\frac{p^2}{(s-1)d^2}.
\end{equation}
We will now prove a lower bound for $p^2/(s-1)d^2$. By (\ref{eqndbound}), 
$$d(s-1) \le \frac{n-1}{a_1}+\frac{(s-2)(s-1)}{a_1}+s-1<\frac{n}{a_1}+\frac{s^2}{a_1}+s,$$
and this, $a_1 \le 10 \sqrt{sn}$, and (\ref{s/n bound}) yield
\begin{equation} \label{p^2/d^2bound}
\begin{split}
\frac{p^2}{(s-1)d^2}
& \ge \frac{(n/(s-1))^2}{(s-1)d^2} \\
& = \frac{n^2}{s-1} \frac{1}{d^2(s-1)^2} \\
& \ge \frac{n^2}{s-1}\left(\frac{n}{a_1}+\frac{s^2}{a_1}+s \right)^{-2} \\
& = \frac{a_1^2}{s-1} \left(1+\frac{s^2}{n}+\frac{a_1 s }{n} \right)^{-2} \\
& \ge \frac{a_1^2}{s-1} \left( 1+ \frac{s^2}{n}+\frac{10s^{3/2}}{\sqrt{n}}\right)^{-2} \\
& \ge \frac{a_1^2}{s-1}(1+\delta_2)^{-2}.
\end{split}
\end{equation}
 Combining (\ref{p^2/d^2bound}) and (\ref{mlowerbound}), we get
 \begin{equation} \label{simplifiedmbound}
     m \ge \frac{1-\delta_2}{(1+\delta_2)^2} \frac{a_1^2}{s-1}.
 \end{equation}
{\bf Case 2.1:} $a_1 \ge (1+\delta_1)\sqrt{n}$.

By (\ref{simplifiedmbound}), 
$$m \ge (1+\delta_1)^2 \frac{1-\delta_2}{(1+\delta_2)^2}\frac{n}{s-1}.$$
Since $\delta_1 \gg \delta_2$, we have $(1+\delta_1)^2 (1-\delta_2)/(1+\delta_2)^2>1$. It follows that $m>(n-1)/(s-1)+s-1$ for $n \ge n_0(s)$.

{\bf Case 2.2:} $a_1 < (1+\delta_1)\sqrt{n}$.

We can assume that there are at most $n/(s-1)$ sets in $\mathcal{L} \setminus N(v)$ as otherwise  $m \ge n/(s-1)+d>(n-1)/(s-1)+s-1$ by  (\ref{dlowerbound}).  Recall that $\delta_3=\delta^{1/4}$ and $\delta_4 = \delta^{1/2}$.

\begin{claim} \label{claimcover}
At least $ (1-\delta_3)n/(s-1)$ sets in $\mathcal{L} \setminus N(v)$ have at least $(1/(s-1)-\delta_4)\sqrt{n}$ points in $P$.
\end{claim}

\begin{proof} Assume this is not true. Since every   $A \in \mathcal{L} \setminus N(v)$ has at most one point in common with every set in $N(v)$, we conclude that $|A\cap P|\le d$.
Hence the number of covered pairs in $P$  is at most 
\begin{equation} \label{pairsinP}
    (1-\delta_3)\frac{n}{s-1}\frac{d^2}{2}+\frac{\delta_3 n}{s-1}\left(\frac{1}{s-1}-\delta_4 \right)^2 \frac{n}{2}+\sum_{j=1}^d \binom{a_{i_j}}{2}.
\end{equation}
Recall that $d \le \sqrt{n}/((1-\delta)(s-1))+1+\delta_2$ by (\ref{eqndboundsimplified}). By this bound and (\ref{s/n bound}) we have
\begin{equation} \label{firstterm}
    \begin{split} 
        (1-\delta_3)\frac{n}{(s-1)}\frac{d^2}{2} 
        & \le (1-\delta_3) \frac{n}{2(s-1)}\left(\frac{\sqrt{n}}{(1-\delta)(s-1)}+\delta_2+1 \right)^2 \\
        & = (1-\delta_3)\frac{n}{2(s-1)} \frac{n}{(1-\delta)^2(s-1)^2}\left(1+\frac{(\delta_2+1)(1-\delta)(s-1)}{\sqrt{n}}\right)^2 \\
        & \le \frac{1-\delta_3}{(1-\delta)^2}\frac{n^2}{2(s-1)^3}(1+\delta_2)^2.
    \end{split}
\end{equation}
We also have
\begin{equation} \label{secondterm}
    \frac{\delta_3 n}{s-1}\left( \frac{1}{s-1}-\delta_4\right)^2 \frac{n}{2} =\frac{\delta_3 n^2}{2(s-1)^3}  \left(1-\delta_4 (s-1)\right)^2.
\end{equation} 
Since $a_1<(1+\delta_1)\sqrt{n}$ and $d \le 2 \sqrt{n}/((1-\delta)(s-1))$ by (\ref{eqndbound}), we have
\begin{equation} \label{thirdterm}
    \begin{split}
        \sum_{j=1}^d \binom{a_{i_j}}{2} 
        & \le \frac{a_1^2 }{2} d\\
        & \le \frac{(1+\delta_1)^2 n }{2}  \frac{2\sqrt{n}}{(1-\delta)(s-1)} \\
        & \le \frac{n^2}{2(s-1)^3} \frac{2(s-1)^2(1+\delta_1)^2}{(1-\delta)\sqrt{n}} \\
        & \le \delta_2 \frac{n^2}{2(s-1)^3}.
    \end{split}
\end{equation}
Note that we used (\ref{s/n bound}) in the last step. Combining (\ref{firstterm}), (\ref{secondterm}), and (\ref{thirdterm}), we deduce that the number of covered pairs in $P$ is at most 
\begin{equation} \label{pairsP}
    \frac{n^2}{2(s-1)^3} \left(\frac{(1+\delta_2)^2(1-\delta_3)}{(1-\delta)^2}+\delta_3\left( 1-\delta_4 (s-1)\right)^2+\delta_2\right).
\end{equation}
As $\delta \gg \delta_2$, we obtain 
\begin{align*}
     \frac{(1+\delta_2)^2(1-\delta_3)}{(1-\delta)^2} 
     & \le (1+3\delta_2)(1+3\delta)(1-\delta_3) \\
     & \le (1+4\delta)(1-\delta_3) \\
     & = 1 - \delta_3 (1+4\delta-4\delta / \delta_3).
\end{align*}
 As $\delta_3 \delta_4 = \delta^{1/4}\delta^{1/2} =\delta^{3/4} \gg \delta$, we have $4\delta +\delta_4(s-1)/2
 >4\delta/\delta_3$ and hence 
$$1 - \delta_3 (1+4\delta-4\delta / \delta_3) \le 1-\delta_3(1-\delta_4(s-1)/2).$$ We also have $\delta_3\left( 1-\delta_4 (s-1)\right)^2 \le \delta_3\left( 1-\delta_4 (s-1)\right)$. It follows that (\ref{pairsP}) is at most 
\begin{align*} & \frac{n^2}{2(s-1)^3}\left(1-\delta_3\left(1-\frac{\delta_4(s-1)}{2}\right)+\delta_3\left( 1-\delta_4 (s-1)\right)+\delta_2\right) \\ \notag
& \le \frac{n^2}{2(s-1)^3}\left(1-\frac{\delta_3 \delta_4 (s-1)}{2}+\delta_2 \right). \notag \end{align*}
Note that $1-\delta_3 \delta_4 (s-1)/2+\delta_2<1$ as $\delta_3 \delta_4 \gg \delta \gg \delta_2$. Since $p^2/2(s-1) \ge n^2/2(s-1)^3$, this implies that for $n \ge n_0(s)$ the number of covered pairs in $P$ is less that $p^2/2(s-1)-p/2$. This contradiction completes the proof of the claim.
\end{proof}
Suppose $A \in \mathcal{L} \setminus N(v)$. If $A$ has at least $(1/(s-1)-\delta_4) \sqrt{n}$ points in $P$, then it has at most $(1+\delta_1-1/(s-1)+\delta_4)\sqrt{n}$ points in $Q$ as $a_1<(1+\delta_1)\sqrt{n}$. Hence, by Claim \ref{claimcover} the number of covered pairs in $Q$ is at most 
$$\frac{\left(1-\delta_3\right)n}{s-1}\binom{(1+\delta_1-1/(s-1)+\delta_4)\sqrt{n}}{2} + \frac{\delta_3 n}{s-1} \binom{(1+\delta_1)\sqrt{n}}{2}$$
\begin{equation} \label{pairs in Q}
    \le \frac{1}{2} \left(1+\delta_1+\delta_4-\frac{1}{s-1} \right)^2 \frac{(1-\delta_3)n^2}{s-1}+\frac{(1+\delta_1)^2\delta_3n^2}{2(s-1)}.
\end{equation}
Note that by (\ref{eqndbound}) 
\begin{equation} \label{Qlowerbound}
\begin{split}
    |Q| 
& \ge n-d(a_1-1)-1 \\
& \ge n-da_1 \\
& \ge n - \left(\frac{n-1}{s-1}+s-2+a_1 \right) \\
& = \left(1-\frac{1}{s-1}\right)n +\frac{1}{s-1}-s+2-a_1.
\end{split}
\end{equation}
Since every $(s-1)$-set in $Q$ is covered, at least $|Q|^2/2(s-2)-|Q|/2$ pairs in $Q$ must be covered. We now show that
\begin{equation} \label{coveredpairsQ}
    \frac{1}{2} \left(1+\delta_1+\delta_4-\frac{1}{s-1} \right)^2 \frac{(1-\delta_3)}{s-1}+\frac{(1+\delta_1)^2\delta_3}{2(s-1)}<\frac{1}{2(s-2)}\left(1-\frac{1}{s-1}\right)^2.
\end{equation}
To see this, first note that $(1+\delta_1)^2\delta_3/2(s-1) \le 1/100s^3$ as $\delta_3 \ll 1/s^2$. Next
  \begin{align*}
&\frac{1}{2} \left(1+\delta_1+\delta_4-\frac{1}{s-1} \right)^2 \frac{(1-\delta_3)}{s-1}\\
& \le \frac{1}{2(s-1)} \left( \left(1-\frac{1}{s-1} \right)^2+3(\delta_1+\delta_4) \right)(1-\delta_3) \\
& \le \frac{1}{2(s-1)}\left(1-\frac{1}{s-1} \right)^2-\frac{\delta_3}{2(s-1)} \left(1-\frac{1}{s-1} \right)^2 + 3(\delta_1+\delta_4) \\
& \le \frac{1}{2(s-1)}\left(1-\frac{1}{s-1} \right)^2 -\frac{\delta_3}{3s} \\
& =  \frac{1}{2(s-2)}\left(1-\frac{1}{s-1} \right)^2 -\frac{1}{2(s-1)(s-2)}\left(1-\frac{1}{s-1} \right)^2-\frac{\delta_3}{3s}
\end{align*}
as $\delta_3 \gg \delta_4 \gg \delta_1$. Since $[1/2(s-2)(s-1)] \cdot (1-1/(s-1))^2 > 1/100s^3$, this proves (\ref{coveredpairsQ}). This means the quadratic coefficient of  the lower bound of $|Q|^2/2(s-2)-|Q|/2$ from (\ref{Qlowerbound}) is larger than the quadratic coefficient of the upper bound for the number of covered pairs in $Q$ from (\ref{pairs in Q}). It follows that for $n \ge n_0(s)$, the number of covered pairs in $Q$ is less than $|Q|^2/2(s-2)-|Q|/2$. Hence we have a contradiction, so it is not possible that $(1-\delta)\sqrt{n}\le a_1 < (1+\delta_2)\sqrt{n}$.

{\bf Case 3:} $10 \sqrt{sn} \le a_1 \le \left( \frac{1}{s-1}-\varepsilon \right) n$, where $\varepsilon=1/10s^2$. 

 By Lemma~\ref{bounds}.5, we can assume $d \in \left\{q,q+1\right\}$. Suppose $d=1$. Then, $|Q|=n-a_1 \ge (1-1/(s-1)+\varepsilon)n$, so by Lemma \ref{bounds}.2 for $n \ge n_0(s)$
 $$m > \frac{|Q|-1}{s-2}+s-2 \ge \frac{n}{s-1}+\frac{\varepsilon n}{s-2}+s-2-\frac{1}{s-2}>\frac{n-1}{s-1}+s-1.$$
 Hence, we can assume $d \ge 2$. If $|Q| > n \left(1- 1/(s-1)\right)$, then for $n \ge n_0(s)$
 \begin{equation} \label{Qbound}
     m \ge \frac{|Q|-1}{s-2}+s-2+d \ge \frac{n}{s-1}+s-2-\frac{1}{s-2}+d>\frac{n-1}{s-1}+s-1.
 \end{equation}
 by Lemma~\ref{bounds}.2.
 Hence, we can assume $p \ge n/(s-1)$. We consider the cases $d \ge s+2$ and $d \le s+1$ separately. Suppose $d \ge s+2$. Observe that 
$$d \le q+1 \le  \frac{n-1}{a_1(s-1)}+\frac{s-2}{a_1}+1$$
and 
$$\frac{n}{s-1} \le p \le a_1 d \le \frac{n-1}{s-1}+s-2+a_1.$$
Suppose there are three sets in the neighborhood of $v$ with size less than $a_1/2$. Then,
$$p \le a_1(d-3)+\frac{3a_1}{2} \le \frac{n-1}{s-1}+s-2-\frac{a_1}{2}.$$
Since $a_1 \ge 10 \sqrt{sn}$, this upper bound for $p$ is smaller than $n/(s-1)$, so we have a contradiction. Hence there are at most two sets in the neighborhood of $v$ with size less that $a_1/2$. As we are assuming $d \ge s+2$,  there are at least $s$ sets with size more than $a_1/2 -1 \ge 4\sqrt{sn} $. Taking disjoint subsets of these $s$ sets of size $4\sqrt{sn}$ and applying Lemma~\ref{bounds}.3 we get
$$m \ge \frac{(4\sqrt{sn})^{s}}{\binom{s}{2} (4\sqrt{sn})^{s-2} } 
=\frac{16sn}{\binom{s}{2} } = \frac{32n}{s-1} > \frac{n-1}{s-1}+s-1$$
for $n \ge n_0(s)$. \\  
We now consider the case in which $d \le s+1$. Since $n/(s-1) \le p \le a_1d$, we have
$$a_1 \ge \frac{n}{d(s-1)} \ge \frac{n}{s^2-1}.$$
Let $A_1=\left\{x_1,\ldots,x_{a_1} \right\}$. If there are at least $2n/s^3$ vertices in $A_1$ with degree at least $s^2+1$, then $m>2n/s>(n-1)/(s-1)+s-1$, so we can assume that there are at most $2n/s^3$ vertices in $A_1$ with degree at least $s^2+1$. This implies that there are at least $n/(s^2-1)-2n/s^3$ vertices in $A_1$ with degree at most $s^2$. For $1 \le i \le a_1$, let
$$T_i=\sum_{j: \, x_i \in A_j, j\neq 1} \binom{a_j-1}{2}$$
and let $B=\left\{i : d(x_i) \le s^2 \right\}$. Then we have 
$$\sum_{i \in B} T_i < n^2$$
as each pair is in at most one $A_i$. It follows that there is some $\ell \in B$ such that 
$$T_\ell \le \frac{n^2}{|B|} \le \frac{n^2}{\frac{n}{s^2-1}-\frac{2n}{s^3}}=\frac{s^3(s^2-1)}{s^3-2s^2+2}n\le 4s^2 n.$$ 
By Jensen's inequality and the inequality $\binom{x}{2} \ge x^2/16$,
$$T_{\ell} \ge (d(x_{\ell})-1) \binom{\frac{1}{d(x_{\ell})-1} \sum_{k: \, x_{\ell} \in A_k, k\neq 1} (a_k-1)}{2} \ge \frac{1}{16(d(x_{\ell})-1)}\left(\sum_{k: \, x_{\ell} \in A_k, k\neq 1} (a_k-1) \right)^2.$$
Comparing the lower bound and upper bound for $T_{\ell}$ yields
$$\sum_{k: \, x_{\ell} \in A_k, k\neq 1} (a_k-1) \le 4\sqrt{d(x_{\ell})-1} \cdot 2s \sqrt{n}\le 8s^2 \sqrt{n}.$$

Let $Q_1$ be the set of points outside of the neighborhood of $x_{\ell}$. Then every $(s-1)$-set in $[n] \setminus Q_1$ is covered. Furthermore, since $a_1 \le (1/(s-1)-\varepsilon)n$ and $\varepsilon>\delta_2$
\begin{align*}
|Q_1| \ge n-a_1- 8s^2 \sqrt{n}
& \ge \left(1- \frac{1}{s-1}+\varepsilon-\frac{8s^2}{\sqrt{n}}\right)n \\
& \ge \left(1-\frac{1}{s-1}+\varepsilon-\delta_2\right)n \\
& > \left(1- \frac{1}{s-1} \right)n,
\end{align*}
so by Lemma~\ref{bounds}.2 with $Q$ replaced with $Q_1$ we get $m > (n-1)/(s-1)+s-1$ by the same computation as (\ref{Qbound}). 

{\bf Case 4:} $\left(1/(s-1)-\varepsilon \right)n<a_1<\lfloor (n-1)/(s-1) \rfloor$.

Suppose $d=1$. Then $|Q|=n-a_1>(s-2)n/(s-1)$ as $a_1 \le (n-1)/(s-1)-1$, so by Lemma \ref{bounds}.2 we have 
$$m \ge \frac{n-a_1-1}{s-2}+s-2+1>\frac{n-1}{s-1}+s-1.$$ 
We can assume $d=2$ as if $d \ge 3$, then $m>2a_1>(2/(s-1)-2\varepsilon)n>(n-1)/(s-1)+s-1$. Furthermore, we can assume the number of vertices in $A_1$ with degree greater than two is at most $(\varepsilon+1/s^3) n$ as if not the number of sets that intersect $A_1$ is at least
$$\left(\frac{1}{s-1}-\varepsilon \right)n+\left(\varepsilon+\frac{1}{s^3}\right)n=\left(\frac{1}{s-1}+\frac{1}{s^3}\right)n>\frac{n-1}{s-1}+s-1.$$
Suppose all the $(s-1)$-sets in $[n] \setminus A_1$ are covered. Observe that
$$|[n] \setminus A_1|=n-a_1 >n-\frac{n-1}{s-1}>\frac{s-2}{s-1}n>n_0(s-1),$$
and $a_1 \le (|Q|-1)/(s-2)$ by the same inequality used in the proof of Lemma  by \ref{bounds}.2. Define
$$\cL^{\prime} := \left\{A \cap ([n] \setminus A_1): A \in \cL, \,  |A \cap ([n] \setminus A_1)|\ge2 \right\}.$$
By induction on $s$
$$|\cL^{\prime}| \ge \frac{n-a_1-1}{s-2}+s-2 > \frac{n-1}{s-1}+s-2.$$
Since $A_1 \cap ([n] \setminus A_1)=\emptyset$, we have
$$m >\frac{n-1}{s-1}+s-1.$$ 
Hence we can assume there is some $(s-1)$-set $x_1,x_2,\ldots, x_{s-1}$ in $[n] \setminus A_1$ that is not covered. For any $p \in A_1$ the $s$-set $\{p,x_1,x_2,\ldots, x_{s-1}\}$ is covered, so there is a set containing a pair $px_i$ for some $i \in [s-1]$. Set
$$B_{x_i}=\left\{p \in A_1: p,x_i \in A_j \text{ for some } j\right\}$$
for $1 \le i \le s-1$. Without loss of generality, assume $|B_{x_1}| \ge |B_{x_2}| \ge \ldots \ge |B_{x_{s-1}}|$. Then 
$$|B_{x_1}| \ge \frac{a_1}{s-1}>\left(\frac{1}{(s-1)^2}-\frac{\varepsilon}{s-1} \right)n.$$
Let $B_{x_1}^{\prime} \subset B_{x_1}$ be the points in $B_{x_1}$ 
that have degree two. Since the number of points in $A_1$ with degree greater than two is at most $(\varepsilon+1/s^3)n$, we have 
$$|B_{x_1}^{\prime}| \ge \left( \frac{1}{(s-1)^2}-\frac{1}{s^3}-\varepsilon \left(1+\frac{1}{s-1} \right) \right)n.$$
Let $Q_2 = [n] \setminus (A_1 \cup \left\{x_1\right\})$ (see Figure \ref{x_1,...,x_{s-1}}). Suppose $\left\{u_1,\ldots, u_{s-1} \right\} \subset Q_2$ is uncovered and let $p \in B_{x_1}^{\prime}$. Then the $s$-set $\{p,u_1,u_2,\ldots, u_{s-1}\}$ must be covered, so there a set $A_i$ containing  $p,u_i$ for some $i$. Since $|B_{x_1}^{\prime}|>s-1$, there is $p_1,p_2 \in B_{x_1}^{\prime}$ and $1 \le j \le s-1$ so that the pairs $p_1 u_j$ and $p_2 u_j$ are both covered. The sets containing these pairs are distinct as $p_1,p_2 \in A_1$. This is a contradiction as any set containing a point from $B_{x_1}$ contains $x_1$, so it implies that pair $u_j x_1$ is in two distinct sets. Hence, all $(s-1)$-sets in $Q_2$ are covered. Observe that since $a_1 \le (n-1)/(s-1)$,
\begin{equation} \label{Q_2>n_0}
 |Q_2| = n-a_1-1  \ge n- \frac{n-1}{s-1}-1 =\frac{s-2}{s-1}n+\frac{1}{s-1}-1 \ge n_0(s-1).
\end{equation}

Consider the collection of sets $\left\{A_i \cap Q_2 \right\}$. Since $a_1 \le (n-1)/(s-1)-1\le (n-2)/(s-1)$, 
$$a_1 \le \frac{n-a_1-2}{s-2} =\frac{|Q_2|-1}{s-2},$$
so $|A_i \cap Q_2| \le (|Q_2|-1)/(s-2)$ for all $i$. Suppose $a_1 < (n-1)/(s-1)-1$. Then, by induction on $s$, the number of sets in $\mathcal{L}$ in $Q_2$ is at least 
\begin{equation} \label{Q_2 bound}
    \frac{|Q_2|-1}{s-2}+s-2 = \frac{n-a_1-2}{s-2}+s-2 >  \frac{n-1}{s-1}+s-2,
\end{equation}
Since $A_1 \cap Q_2 = \emptyset$, $m > (n-1)/(s-1)+s-1$. 

We now consider the case  $a_1 = (n-1)/(s-1)-1$. Note that $(s-1) \, | \, (n-1)$ in this case. By the inequality in (\ref{Q_2 bound}), we have $m \ge (n-1)/(s-1)+s-1$. As stipulated by the induction statement, we are required to show that this inequality is strict. 
 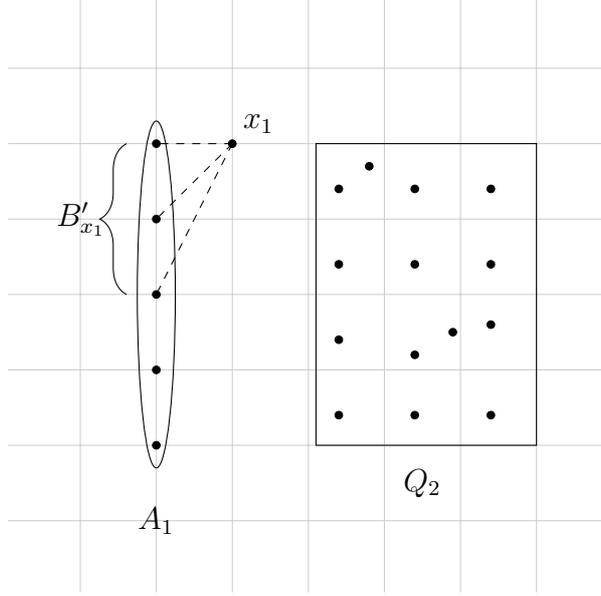
\begin{figure}
   \begin{center}
      \begin{tikzpicture}

        \draw[step=1,help lines,black!20] (-0.95,-0.95) grid (6.95,6.95);
        
        \foreach \Point/\PointLabel in {((1,1)/, (1,2)/, (1,3)/, (1,4)/, (1,5)/,(2,5)/x_1,(3.4,1.4)/,(4.4,1.4)/,(5.4,1.4)/,(3.4,2.4)/,(4.4,2.2)/,(5.4,2.6)/,(3.4,3.4)/,(4.4,3.4)/,(5.4,3.4)/,(3.4,4.4)/,(4.4,4.4)/,(5.4,4.4)/,(3.8,4.7)/,(4.9,2.5)/ }
        \draw[fill=black] \Point circle (0.05) node[above right]
        {$\PointLabel$};
        \draw (1,3) ellipse (0.25cm and 2.3cm) node at (1,0)  {$A_1$};
        \draw[dashed](2,5)--(1,5) node[above]{};
        \draw[dashed](2,5)--(1,4) node[above]{};
        \draw[dashed](2,5)--(1,3) node[above]{};
       \draw [decorate,decoration={brace,amplitude=10pt},xshift=-4pt,yshift=0pt] (0.75,3) -- (0.75,5) node [black,midway,xshift=-0.6cm] {$B_{x_1}^{\prime}$};
        \draw [draw=black] (6,5) rectangle (3.1,1) node at (4.5,0.5)  {$Q_2$};
    \end{tikzpicture}  
    \end{center}
     \caption{$Q_2$ and $B_{x_1}^{\prime}$} \label{x_1,...,x_{s-1}}
   \end{figure}

\begin{claim} \label{a_1=(n-1)/(s-1)-1}
If $a_1 \le (n-1)/(s-1)-1$, then $m>(n-1)/(s-1)+s-1$.
\end{claim}

\begin{proof} If $a_1<(1/(s-1)-\varepsilon)n$, then the claim follows by the arguments in Cases 1-3. If $(1/(s-1)-\varepsilon)n \le a_1<(n-1)/(s-1)-1$, then (\ref{Q_2 bound}) shows $m>(n-1)/(s-1)+s-1$, so we can assume $a_1=(n-1)/(s-1)-1$. Define 
$$\cL_2 := \left\{A \cap Q_2 : A \in \cL, \, |A \cap Q_2| \ge 2 \right\}.$$
We first consider the case $s=3$. Then $a_1=(n-1)/2-1$, $|Q_2|=n-(n-1)/(s-2)+1-1=(n-1)/2+1$, and every pair in $Q_2$ is covered. By the de Bruijn Erd\H os theorem,  $|\cL_2|\ge (n-1)/2+1$. If $|\cL_2|>(n-1)/2+1$, then $m>(n-1)/2+2$ as $A_1 \cap Q_2 = \emptyset$. Hence we can assume $|\cL_2| = (n-1)/2+1$. By the de Bruijn Erd\H os theorem, $\cL_2$ is either a near pencil or projective plane. If $\cL_2$ is a near pencil then, there is a set of size $(n-1)/2$ in $\cL_2$. This is a contradiction as the largest set in $\mathcal L$ has size $a_1 =(n-1)/2-1$. Suppose now that $\cL_2$ is a projective plane, so $|\mathcal L_2|=|Q_2|=(n-1)/2$. Recall that $x_1$ has degree at least $|B_{x_1}|\ge a_1/(s-1)=a_1/2\ge 3$. If $N(x_1)$ contains $A_i, A_j$ such that neither $A_i \cap Q_2$ nor $A_j \cap Q_2$ is in $\mathcal L_2$,  then $m \ge |\mathcal L_2| +3 > (n-1)/2+2$. So $N(x_1)$ contains $A_i, A_j$ such that both $A_i \cap Q_2$ and $A_j \cap Q_2$ are in $\mathcal L_2$. But  $\mathcal L_2$ is a projective plane hence $(A_i \cap Q_2) \cap (A_j \cap Q_2) \ne \emptyset$, which means that $
|A_i \cap A_j| \ge 2$, a contradiction. 

Now, suppose $s \ge 4$. Suppose $a_1=(n-1)/(s-1) -1$. Then every $(s-1)$-set in $Q_2$ is covered. We also have 
$$\frac{|Q_2|-1}{s-2}=\frac{n-a_1-2}{s-2}=\frac{n-1}{s-1},$$
so $a_1 \le (|Q_2|-1)/(s-2)-1$. This shows $|A| \le (|Q_2|-1)/(s-2)-1$ for all $A \in \cL_2$. Since $|Q_2| \ge n_0(s-1)$ by (\ref{Q_2>n_0}),  the inductive hypothesis implies
$$|\cL_2| > \frac{|Q_2|-1}{s-2}+s-2=\frac{n-1}{s-1}+s-2$$
As $A_1 \not\in \mathcal L_2$, it follows that $m>(n-1)/(s-1)+s-1$  and this 
concludes the proof of the claim and Case 4.\end{proof}

{\bf Case 5:} $a_1= \lfloor (n-1)/(s-1) \rfloor $.

We note that for  $x>1$, we have $\lfloor x \rfloor> x-1$ so in this case $a_1 > (n-1)/(s-1)-1$ and we only need to prove that $m \ge (n-1)/(s-1)+s-1$.

Suppose all the $(s-1)$-sets in $[n] \setminus A_1$ are covered. Observe that $n-a_1 \ge n_0(s-1)$ by (\ref{Q_2>n_0}) and $a_1 \le (n-a_1-1)/(s-2)$ as $a_1 \le (n-1)/(s-1)$. Define
$$\cL^{\prime} := \left\{A \cap ([n] \setminus A_1): A \in \cL, \,  |A \cap ([n] \setminus A_1)|\ge2 \right\}.$$
By induction on $s$, 
$$|\cL^{\prime}| \ge \frac{n-a_1-1}{s-2}+s-2\ge \frac{n-1}{s-1}+s-2.$$
Since $A_1 \cap ([n] \setminus A_1) = \emptyset$, we have $m \ge (n-1)/(s-1)+s-1$.

Now, we consider the case in which there is an $(s-1)$-set $\{x_1,\ldots, x_{s-1}\}$ in $[n] \setminus A_1$ that is uncovered. Note that we can assume $d:=\min_{w \in A_1}d(w)=2$ in this case as if $d=1$, then all the $(s-1)$-sets in $[n] \setminus A_1$ are covered and if $d \ge 3$, then $m \ge 2a_1>(n-1)/(s-1)+s-1$. Furthermore, we may assume that the number of points with degree at least 3 in $A_1$ is at most $s-1$, as if there are at least $s$ points with degree at least 3, we have $m \ge 1+a_1+s \ge (n-1)/(s-1)+s-1$. For any $p \in A_1$ the $s$-set $\{p,x_1,x_2,\ldots, x_{s-1}\}$ is covered, so $\{p,x_i\}$ is covered for some $i \in [s-1]$. For  $1 \le i \le s-1$, set
$$B_{x_i}=\left\{p \in A_1: \{p,x_i\} \subset A_j \text{ for some } j\right\}.$$
 Without loss of generality, assume $|B_{x_1}| \ge |B_{x_2}| \ge \cdots \ge |B_{x_{s-1}}|$. Then 
$$|B_{x_1}| \ge \frac{a_1}{s-1}.$$
Let $B_{x_1}^{\prime} \subset B_{x_1}$ be the set of points in $B_{x_1}$ 
that have degree two. Since the number of points in $A_1$ with degree greater than two is at most $s-1$, we have 
\begin{equation} \label{eqn:bx1'}
|B_{x_1}^{\prime}| \ge \frac{a_1}{s-1}-s+1.\end{equation}
Let $Q_2 = [n] \setminus (A_1 \cup \left\{x_1\right\})$. Suppose $\left\{u_1,\ldots, u_{s-1} \right\} \subset Q_2$ is uncovered and let $p \in B_{x_1}^{\prime}$. Then the $s$-set $\{p,u_1,u_2,\ldots, u_{s-1}\}$ must be covered, so there a set $A_{i_j}$ containing  $p$ and $u_j$ for some $j$. Since $|B_{x_1}^{\prime}|>s-1$, there are $p_1,p_2 \in B_{x_1}^{\prime}$ and $1 \le j \le s-1$ so that the pairs $p_1 u_j$ and $p_2 u_j$ are both covered. The sets containing these pairs are distinct as $p_1,p_2 \in A_1$. This is a contradiction as any set containing a point from $B_{x_1}$ contains $x_1$, so it implies that the pair $u_j, x_1$ is in two distinct sets. 
\begin{equation} \label{eqn:covered}
\hbox{ Hence all $(s-1)$-sets in $Q_2$ are covered.} 
\end{equation}
Define 
$$\cL_2 := \left\{A \cap Q_2 : A \in \cL, \, |A \cap Q_2| \ge 2 \right\}.$$
We first consider the case  $s=3$. Suppose first that $n$ is odd, say $n=2k+1$ for some integer $k$. Then, $a_1=k$ and $|Q_2| = k$. Then either $Q_2 \in \cL$ or by the de Bruijn-Erd\H os theorem $|\cL_2|\ge k$. Suppose $Q_2 \in \mathcal{L}$. Since $A_1$ has minimum degree 2, we have $m \ge k+2 = (n-1)/2+2$ as required. 
Now suppose $Q_2 \notin \cL$. By the de Bruijn-Erd\H os theorem, $|\cL_2|\ge k$. If $|\cL_2|\ge k+1$, we have $m \ge k+2$ as $A_1 \cap Q_2 = \emptyset$. If $|\cL_2|=k$, then $\cL_2$ is either a near pencil or projective plane. In either case, any pair of sets in $Q_2$ intersects in exactly one point. Suppose $m=k+1$ for the sake of contradiction. Then since $x_1$ has degree at least $a_1/(s-1)=a_1/2$, there must be a matching of size $a_1/2$ in $Q_2$. This is a contradiction, so we have $m \ge k+2$. 

Now, suppose $n=2k$. Then, $a_1=k-1$ and $|Q_2| =k$. Since $|Q_2|>a_1$, by the de Bruijn-Erd\H os theorem $|\cL_2| \ge k$. This implies $m \ge k+1$. By the same argument as above using the fact that $x_1$ has degree at least $a_1/2$, we have $m \neq k+1$, so $m \ge k+2$. 

We now consider $s \ge 4$. Write $n-1=(s-1)\ell+r$ for integers $\ell,r$ with $0 \le r < s-1$ so that $a_1 =\ell$.  We will show that $m \ge \ell+s \ge (n-1)/(s-1)+s-1$.

{\bf Case 5.1:} $r \ge 2 $.

Note that $|Q_2|=n-a_1-1= (s-2)\ell+r$ and 
$$\frac{|Q_2|-1}{s-2}=\ell+\frac{r-1}{s-2} \ge \ell.$$
For $i>1$,  $|A_i \cap Q_2| \le a_1 =\ell \le (|Q_2|-1)/(s-2)$. Since $|Q_2| \ge n_0(s-1)$, by induction on $s$,
$$|\cL_2| \ge \frac{|Q_2|-1}{s-2}+s-2=\ell + \frac{r-1}{s-2}+s-2.$$
Since $A_1 \cap Q_2 = \emptyset$,
 $$m \ge \ell + \frac{r-1}{s-2}+s-1.$$
This shows that $m \ge \ell +s$ as $r \ge 2$.

{\bf Case 5.2:} $r \in \left\{0,1\right\}$.

{\bf Case 5.2.1:}  There is no $B_1 \in \cL$ of size $\ell$ such that  $B_1 \subset Q_2$.

Suppose $r=1$. Note that $|Q_2|=(s-2)\ell+1$ and 
$$\frac{|Q_2|-1}{s-2}=\ell.$$
We also have $|A_i \cap Q_2| \le \ell-1$ for all $i$. Since $|Q_2| \ge n_0(s-1)$, by induction we have the strict inequality
$$|\cL_2| > \frac{|Q_2|-1}{s-2}+s-2=\ell +s-2.$$
Since $A_1 \cap Q_2 = \emptyset$, this implies $m \ge \ell + s$.

Suppose $r=0$. Note that $\ell = (n-1)/(s-1)$, $|Q_2|=(s-2)\ell$ and 
$$\left\lfloor  \frac{|Q_2|-1}{s-2} \right\rfloor =\ell-1.$$
We also have $|A_i \cap Q_2| \le \ell-1$ for all $i$ and $|Q_2| \ge n_0(s-1)$. By induction on $s$, 
$$|\cL_2| \ge \frac{|Q_2|-1}{s-2}+s-2=\ell - \frac{1}{s-2}+s-2.$$
Since $A_1 \cap Q_2 = \emptyset$, this implies
 $$m \ge \ell - \frac{1}{s-2}+s-1.$$
Since $s \ge 4$, this shows $m \ge \ell +s-1=(n-1)/(s-1)+s-1$.

{\bf Case 5.2.2:}  There exists $B_1 \in \cL$ of size $\ell$ such that  $B_1 \subset Q_2$.

 We use an iterative argument for this case. Recall (\ref{eqn:covered}) that all $(s-1)$-sets in $Q_2$ are covered. Let $2 \le j \le s-2$. Suppose all $(s-j+1)$-sets in $Q_j$ are covered and $B_{j-1} \subset Q_{j}$ is an $\ell$-set in $\mathcal{L}$. Let $Q_{j+1}=Q_{j} \setminus B_{j-1}$ so that $A_1, B_1, \ldots, B_{j-1}$ is a matching. We can assume the minimum degree among points in $B_{j-1}$ is 2 as we did with $A_1$. Similarly we can also assume that the number of points in $B_{j-1}$ with degree at least 3 is at most $s-1$. We can also assume the maximum degree of points in $B_{j-1}$ is at most $s$, as otherwise $m \ge 1+\ell + (s+1-2) = \ell+s$. Recall that $B_{x_1}^{\prime}$ is the set of $v \in A_1$ such that $\{v,x_1\}$ is covered and $d(v)=2$. We first construct a matching $M_{j-1} \subset B_{x_1}^{\prime} \times B_{j-1}$ of covered pairs. For every $v \in B_{x_1}'$, let $X_v \in \cL$ be the set containing $x_1$ and $v$. We will assume at most $s-1$ of the sets $\left\{ X_v \right\}_{v \in B_{x_1}'}$ are disjoint from $B_{j-1}$ as otherwise we have $m \ge \ell + s$ since points in $B_{j-1}$ have minimum degree 2. This means that there are at least $B_{x_1}'-s+1$ covered pairs in $B_{x_1}^{\prime} \times B_{j-1}$. Since the maximum degree of points in $B_{j-1}$ is at most $s$, there is a matching $M \subset B_{x_1}^{\prime} \times B_{j-1}$ of size at least $(|B_{x_1}^{\prime}|-s+1)/s$. Let $M_{j-1}$ be the matching formed by deleting pairs $(g,h)$ with $d(h) \ge 3$ from $M$. Note that  (\ref{eqn:bx1'}) implies
$$|M_{j-1}|\ge \frac{|B_{x_1}^{\prime}|-s+1}{s}-s+1 >s-j.$$
 We now claim that the $(s-j)$-sets in $Q_{j+1}$ are covered. Suppose $\{h_1,\ldots,h_{s-j}\}$ is an uncovered set in $Q_{j+1}$. Since every $(s-j+1)$-set in $Q_j$ is covered, for every $p \in B_{j-1}$, the $(s-1)$-set $p,h_1,\ldots,h_{s-j}$ is covered. It follows that the pair $p,h_i$ is covered for some $i \in [s-j]$. Since $|M_{j-1}|>s-j$, there is $h_k$ and points $x,y$ in the restriction of $M_{j-1}$ to $B_{j-1}$ so that the pairs $x,h_k$ and $y,h_k$ are both covered. This is a contradiction as $d(x)=d(y)=2$, so sets in $\mathcal{L}$ that contain $x,h_k$ and $y,h_k$ both contain $x_1$. Hence the pair $h_k,x_1$ is in two sets in $\mathcal{L}$. Consequently, all $(s-j)$-sets in $Q_{j+1}$ are covered. Define 
$$\cL_{j+1} := \left\{A \cap Q_{j+1} : A \in \cL, \, |A \cap Q_{j+1}| \ge 2 \right\}.$$ 
 
Let us first suppose that $r=1$. Assume  $j<s-2$. If $|A_i \cap Q_{j+1}| \le \ell-1$ for all $i$, then $|A_i \cap Q_{j+1}| < \ell = (|Q_{j+1}|-1)/(s-j-1)$. Since
$$|Q_{j+1}|=(s-j-1)\ell+1=(s-j-1)\frac{n-2}{s-1} \ge n_0(s-j),$$
and by induction,
$$|\cL_{j+1}| > \frac{|Q_{j+1}|-1}{s-j-1}+s-j-1=\ell+s-j-1.$$
Since $A_1,B_1,\ldots,B_{j-1}$ are disjoint from $Q_{j+1}$, we get $m \ge \ell+s$. Otherwise, there is a set $B_j \subset Q_{j+1}$ with size $\ell$ in $\cL$ and we continue the procedure.  
\begin{figure}
   \begin{center}
      \begin{tikzpicture}

        \draw[step=1,help lines,black!20] (-0.95,-0.95) grid (6.95,6.95);
        
        \foreach \Point/\PointLabel in {((1,1)/, (1,2)/, (1,3)/, (1,4)/, (1,5)/,(2,5)/x_1,(3.2,1.6)/,(3.2,2.4)/,(3.2,3)/,(3.2,3.8)/,(3.2,4.6)/,(4.2,1.6)/,(4.2,2.4)/,(4.2,3)/,(4.2,3.8)/,(4.2,4.6)/,(5.2,1.5)/,(5.2,2.3)/,(5.4,2.9)/,(5.7,3.6)/,(6.2,3.6)/,(6.5,1.6)/,(5.25,5.2)/,(5.25,5.2)/,(6,5)/}
        \draw[fill=black] \Point circle (0.05) node[above right]
        {$\PointLabel$};
        \draw (1,3) ellipse (0.25cm and 2.3cm) node at (1,0)  {$A_1$};
        \draw (3.2,3.1) ellipse (0.25cm and 2cm) node at (3.2,0.6)  {$B_1$};
        \draw (4.2,3.1) ellipse (0.25cm and 2cm) node at (4.2,0.6)  {$B_2$};
        \draw (4.75,1)--(4.75,5.5) node[above]{};
        \draw[dashed](2,5)--(1,5) node[above]{};
        \draw[dashed](2,5)--(1,4) node[above]{};
        \draw[dashed](2,5)--(1,3) node[above]{};
       \draw [decorate,decoration={brace,amplitude=10pt},xshift=-4pt,yshift=0pt] (0.75,3) -- (0.75,5) node [black,midway,xshift=-0.6cm] {$B_{x_1}^{\prime}$};
        \draw [draw=black] (7,5.5) rectangle (2.7  ,1) node at (5.75,0.6)  {$Q_4$};
    \end{tikzpicture}  
    \end{center}
     \caption{Iterative Procedure} \label{B_i}
   \end{figure}
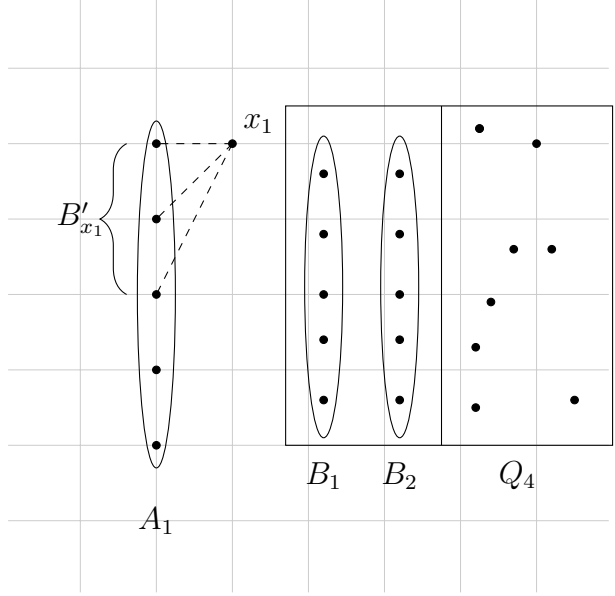
Now, suppose the procedure terminates at $j=s-2$. Then, we have $\ell$-sets $A_1,B_1,\ldots,B_{s-3}$ and all the pairs in $Q_{s-1}=[n] \setminus (A_1 \cup \left\{x_1\right\}\cup B_1 \cup \ldots \cup B_{s-3})$ are covered. Since $|Q_{s-1}|=\ell+1$, by the de Bruijn-Erd\H os theorem, $|\cL_{s-1}|\ge \ell+1$. Since $A_1,B_1,\ldots,B_{s-3}$ are disjoint from $Q_{s-1}$, we have $m \ge \ell+s-1$. We now show that $m \ge \ell+s$. Suppose $m=\ell+s-1$ for the sake of contradiction. Then, $\cL_{s-1}$ is either a near pencil or projective plane. Since $x_1$ has degree at least $\ell /(s-1)$ and $m=\ell+s-1$, there must be a matching of size $\ell/(s-1)$ in $\cL_2$. Since $B_1,\ldots,B_{s-3}$ have size $\ell$, they cannot contain $x_1$. It follows that there must be a matching of size $\ell/(s-1)$ among the elements of $\mathcal{L}$ that cover the pairs in $Q_{s-1}$. This is a contradiction as any 2 sets in $\cL_{s-2}$ intersect in exactly one point. This shows $m \ge \ell+s$.

Let us now suppose that $r=0$. Assume $j<s-2$.
If $|A_i \cap Q_{j+1}| \le \ell-1$ for all $i$, then $|A_i \cap Q_{j+1}| \le \lfloor (|Q_{j+1}-1)/(s-j-1) \rfloor$ due to
$$\frac{|Q_{j+1}|-1}{s-j-1}=\ell-\frac{1}{s-j-1}.$$
Since 
$$|Q_{j+1}|=(s-j-1)\ell=(s-j-1) \cdot \frac{n-1}{s-1}\ge n_0(s-j),$$
we may apply induction to obtain
$$|\cL_{j+1}| \ge \frac{|Q_{j+1}|-1}{s-j-1}+s-j-1=\ell-\frac{1}{s-j-1}+s-j-1.$$ Since $A_1,B_1,\ldots,B_{j-1}$ are disjoint from $Q_{j+1}$, we have $m \ge \ell+s-1=(n-1)/(s-1)+s-1$ as $s-j-1>1$. Otherwise, there is a set $B_j \subset Q_{j+1}$ with size $\ell$ in $\cL$ and we continue the procedure.

Now, suppose the procedure goes up to $j=s-2$. Then we have $\ell$-sets $A_1,B_1,\ldots,B_{s-3}$ and all the pairs in $Q_{s-1}=[n] \setminus (A_1 \cup \left\{x_1\right\}\cup B_1 \cup \ldots \cup B_{s-3})$ are covered. Since $|Q_{s-1}|=\ell$, either $Q_{s-1} \in \mathcal{L}$ or by the de Bruijn-Erd\H os theorem, $|\cL_{s-1}|\ge \ell$. If $Q_{s-1} \in \mathcal{L}$, then $m \ge \ell+s-1$ as we have $s-1$ sets of size $\ell$ and an additional $\ell$ sets from the fact that $A_1$ has minimum degree 2. If $Q_{s-1} \notin \mathcal{L}$, then $|\cL_{s-1}|\ge \ell$, so we have $m \ge \ell+s-2$. Suppose $m=\ell+s-2$. Since $x_1$ has degree at least $\ell /(s-1)$ and $m=\ell+s-2$, there must be a matching of size $\ell/(s-1)$ in $\cL_2$. Since $B_1,\ldots,B_{s-3}$ have size $\ell$, they cannot contain $x_1$. It follows that there must be a matching of size $\ell/(s-1)$ in the elements of $\mathcal{L}$ that cover the pairs in $Q_{s-1}$. This is a contradiction as any two sets in $\cL_{s-1}$ intersect in exactly one point. This shows $m \ge \ell+s-1=(n-1)/(s-1)+s-1$. This concludes the proof of the lower bound in Theorem~\ref{mainthm}. \qed
\bigskip

{\bf Proof of Tightness for Theorem~\ref{mainthm}.} 
Suppose $(s-1) \, | \, (n-1)$. We construct a family of $s$ covers $\mathcal{L}_n(s)$ on $[n]$ with size $(n-1)/(s-1)+s-1$. Recall that a near pencil on $[n]$  comprises $n$ sets $A, B_1, \ldots, B_{n-1}$ where $A=[n-1]$ and $B_i=\{i,n\}$. By the de Bruijn-Erd\H os theorem, $\cL_n(2)$ consists of the near pencil or a finite
projective plane. For $s \ge 3$, and $t=(n-1)/(s-1)$, let $\cL_n(s)$ consists of families obtained by taking the disjoint union of some $\cL \in 
\cL_{n-t}(s-1)$ with a $t$-set $T$, and possibly enlarging each set in $\cL$ by a point in $T$ while ensuring that no two sets in our family have more than one point in common. It clear that members of $\cL_n(s)$ are $s$-covers as any $s$-set contains either 2 points in $T$ or $s-1$ points in the $\cL \in \cL_{n-t}(s-1)$. Moreover,
$$|\mathcal L|+1 = \frac{n - t-1}{s-2}+(s-2) +1 =\frac{n-1}{s-1} +s-1.$$ 
This shows Theorem~\ref{mainthm} is tight if $(s-1) \, | \, (n-1)$.

We now show Theorem~\ref{mainthm} is tight asymptotically. Suppose $(s-1) \nmid  (n-1)$ and $n$ is sufficiently large in terms of $s$. We construct an $s$-cover of size $n/(s-1) \cdot (1+o(1))$ as $n \rightarrow \infty$. In~\cite{prime}, Baker, Harman, and Pintz showed there is a prime number in the interval $[x,x+x^{0.525}]$ for $x$ sufficiently large. Setting $x=\sqrt{n/(s-1)}$ implies that there is a prime $q$ such that
\begin{equation} \label{qprime}
    \sqrt{\frac{n}{s-1}} \le q \le \sqrt{\frac{n}{s-1}}+\left(\frac{n}{s-1}\right)^{0.2625}.
\end{equation}
Let $A_1,A_2,\ldots, A_{s-3}$ be pairwise disjoint sets of size $\lfloor (n-1)/(s-1) \rfloor$. Let $x = (s-3)\lfloor (n-1)/(s-1) \rfloor+(q^2+q+1)$ and $A_{s-2}$ be a set of size $n-x$ with no points from the previous $A_i$'s. Let $A_{s-1},\ldots, A_m$ be a projective plane on the remaining $q^2+q+1$ points. Let $\cL=\left\{A_1,\ldots,A_m\right\}$. Note that any $s$-set has either 2 points in some $A_i$ where $1 \le i \le s-2$ or 2 points in the projective plane formed by $A_{s-1},\ldots,A_m$, so all $s$-sets are covered. We now show that $|A_{s-2}|=n-x \le \lfloor (n-1)/(s-1) \rfloor$. This is equivalent to 
$$n-(s-2)\left\lfloor \frac{n-1}{s-1}\right\rfloor  \le q^2+q+1.$$
This holds since 
$$n-(s-2)\left\lfloor \frac{n-1}{s-1}\right\rfloor
\le \frac{n-1}{s-1}+s-1 < \frac{n}{s-1}+\sqrt{\frac{n}{s-1}}+1 \le q^2+q+1.$$
 For $s-1 \le t \le m$, 
$$|A_t|=q+1\le \sqrt{\frac{n}{s-1}} +\left(\frac{n}{s-1}\right)^{0.2625}+1< \frac{n-1}{s-1}.$$
This shows $\cL$ is an $s$-cover of size 
$$q^2+q+1+s-2=\frac{n}{s-1}\left(1+o(1)\right)$$
by (\ref{qprime}). \qed

\bibliographystyle{plain}
\bibliography{scover-refs}
\end{document}